\documentclass[12pt,espf]{amsart}
\usepackage{a4}
\usepackage{amssymb}
\usepackage{amsmath}
\usepackage{amsthm}
\usepackage{amstext}
\usepackage{amscd}
\usepackage{latexsym}
\usepackage{graphics}
\usepackage{color}
\input{diagrams.sty}
\usepackage{url}

\theoremstyle{plain}
\newtheorem{thm}{Theorem}[section]
\newtheorem{theorem}[thm]{Theorem}

\newtheorem{lemma}[thm]{Lemma}

\newtheorem{corollary}[thm]{Corollary}
\newtheorem{prop}[thm]{Proposition}

\theoremstyle{definition}

\newtheorem{rmk}[thm]{Remark}

\newtheorem{defn}[thm]{Definition}

\newtheorem{definition}[thm]{Definition}
\newtheorem{Definition-Proposition}[thm]{Definition-Proposition}
\newtheorem{construction}[thm]{Construction}

\newcommand{\Bir}{{\rm Bir}}

\renewcommand{\tilde}{\widetilde}

\newcommand{\A}{{\mathbb A}}

\newcommand{\C}{{\mathbb C}}

\newcommand{\G}{{\mathbb G}}

\renewcommand{\P}{{\mathbb P}}
\newcommand{\Q}{{\mathbb Q}}

\newcommand{\V}{{\mathbb V}}

\input{xy}
\xyoption{all}

\begin{document}
\title{on finiteness of {\bf B}-representations and semi-log canonical abundance}
\dedicatory{ Dedicated to Professor Shigefumi Mori on his 60th Birthday}
\author{Christopher D. Hacon}
\date{\today}
\address{Department of Mathematics \\
University of Utah\\
155 South 1400 East\\
JWB 233\\
Salt Lake City, UT 84112, USA}
\email{hacon@math.utah.edu}
\author{Chenyang Xu}
\address{Beijing International Center of Mathematics Research, 5 Yiheyuan Road, Haidian District, Beijing 100871, China}
\address{Department of Mathematics\\
University of Utah\\
155 South 1400 East\\
Salt Lake City, UT 84112, USA}
\email{dr.chenyang.xu@gmail.com}

\begin{abstract}
We give a new proof of the finiteness of {\bf B}-representations. As a consequence of the finiteness of {\bf B}-representations and Koll\'ar's gluing theory on lc centers, we prove that the (relative) abundance conjecture for slc pairs is implied by the abundance conjecture for log canonical pairs.
\end{abstract}
\thanks{We are indebted to Y. Gongyo, J. M$^{\rm c}$Kernan, J. Koll\'ar and R. Zhong for helpful conversations. Especially, the main ideas in Section 4 follow \cite{Kollar11b}. Part of this work was done while the second author was visiting RIMS, which he would like to thank for the hospitality. The first author was partially supported by NSF research grant no: 0757897, the second author was partially supported by NSF research grant no: 0969495.}
\maketitle


\tableofcontents

\section{Introduction}

Throughout this note, the ground field will be the field  $\C$ of complex numbers. It is well known that even though the log minimal model program is focused on the study of log pairs $(X,\Delta )$ where $X$ is a normal variety, for technical reasons it is often necessary to deal with log pairs $(X,\Delta )$ where $X$ is a semi-normal variety.
This naturally occurs in proofs by induction on the dimension where, for example, we restrict to the reduced part of the boundary of a dlt pair $(X,\Delta )$ (cf. e.g. \cite{Kollar92, KMM94, Birkar11, HX11, FG11}) or when we study moduli of pairs as normal varieties can degenerate to non-normal ones (cf. \cite{Kollar11, HX11}). 
In \cite{Fujino00}, O. Fujino first used the {\bf B}-representation to study semi-log canonical abundance conjecture and proved the conjecture in the 3 dimensional case. 
Recently, J. Koll\'ar has developed a useful technique for gluing log canonical centers (cf. \cite{Kollar11,Kollar11b}) that reduces many questions on semi-normal pairs to questions on their normalizations. 
An important result used in Koll\'ar's theory and in Fujino's work (cf. \cite{Fujino00}), is  the finiteness of {\bf B}-representations, which was first proved by Ueno in the klt case, and then generalized to the log canonical case by Fujino-Gongyo (cf. \cite{FG11}). Recall the following.

\begin{definition}Let $(X,\Delta)$ be a projective dlt pair. We define the birational automorphism group $\Bir(X,\Delta)$ of $(X,\Delta)$ to be the group 
of all birational maps $g$ of $X$ such that if we take a common resolution
\begin{diagram}
 && Y & &\\
 &\ldTo^p & &  \rdTo^q\\
 X& &\rDashto^g & &X
\end{diagram}
then $p^*(K_X+\Delta)=q^*(K_X+\Delta)$. We call the induced homomorphism
$$\rho_{m}:{\rm Bir}(X,\Delta)\to {\rm Aut}(H^0(X,\mathcal O_{X}(m(K_{X}+\Delta))))$$
the {\bf B-representation} of the pair $(X,\Delta )$. As far as we know, $\bf B$-birational maps and
$\bf B$-representations for general log pairs were first explicitly introduced in \cite{Fujino00}.
\end{definition}

In this note we first aim to give a new proof of the following result.

\begin{theorem}\label{t-fb} Given a projective dlt pair $(X,\Delta)$ such that $K_{X}+\Delta$ is a semi-ample $\mathbb{Q}$-divisor.
There exists $m\in \mathbb{N}$, such that the image of the {\bf B}-representation
$$\rho_{M}:{\rm Bir}(X,\Delta)\to {\rm Aut}(H^0(X,\mathcal O _{X}(M(K_{X}+\Delta))))$$
is finite for any positive integer $M$ divisible by $m$.
\end{theorem}
 \begin{rmk} Theorem \eqref{t-fb} was proven by different methods in \cite[1.1]{FG11}.
The argument in the current note was originally contained in \cite{HMX11}. During the preparation of \cite{HMX11} we were informed of \cite{FG11} and we decided to include \eqref{t-fb} in a separated paper. Our proof uses the case when the Kodaira dimension is 0, which was proved by Gongyo (cf. \cite{Gongyo10}) using ideas from \cite{Fujino00}. 
\end{rmk}

The second part of our paper is focused on the study of semi-log canonical abundance. One of the main applications of the finiteness of ${\bf B}$-representations is  to prove that abundance for semi-log canonical pairs, follows from abundance for log canonical pairs (cf. \cite{Fujino00,FG11}). 
Using Koll\'ar's gluing theory, as a consequence of \eqref{t-fb}, we prove the following.
 \begin{theorem}\label{slc}
Let $(X,\Delta)$ be a semi-log canonical pair, $f:X\to S$ a projective morphism, $n: \bar{X}\to X$ the normalization and write $n^*(K_X+\Delta)=K_{\bar{X}}+\bar{\Delta}+\bar{D}$, where $\bar{D}$ is the double locus. If $K_{\bar{X}}+\bar{\Delta}+\bar{D}$ is semi-ample over $S$, then $K_{X}+\Delta$ is semi-ample over $S$.  
\end{theorem}

As a corollary we recover the following result conjectured by C. Birkar (cf. \cite[1.2]{Birkar11}), which is known to be a natural step of log canonical minimal model program in the relative case (cf. \cite[Section 7]{KMM94}).
\begin{corollary}\label{c-lc} Let $(X, \Delta)$ be a $\mathbb{Q}$-factorial dlt pair which is projective over a  variety $S$, and $T := \lfloor \Delta \rfloor$ where $\Delta$ is a $\mathbb{Q}$-divisor. Suppose that
\begin{enumerate} 
\item $K_X + \Delta$ is nef over $S$,  
\item $(K_X + \Delta)|_{T_i}$ is semi-ample over $S$ for each component $T_i$ of $T$, 
\item  $K_X + \Delta-\epsilon P$ is semi-ample over $S$ for some $\mathbb{Q}$-divisor $P \ge 0$ with ${\rm  Supp} (P) = T$ and for any sufficiently small rational number $\epsilon > 0$. 
\end{enumerate}
Then, $K_X + \Delta$ is semi-ample over $S$.
\end{corollary}

Another corollary is the following result, which answers a question raised by J. Koll\'ar in the problem session in MSRI in March 2009.

\begin{corollary}\label{lc}
Let $(X,\Delta)$ be a semi-log canonical pair and $f:X\to S$ a projective morphism, such that $K_X+\Delta\equiv_{\Q , S} 0$. Then $K_X+\Delta\sim_{\mathbb{Q},S}0$.
\end{corollary}

\begin{rmk}
A result due to O. Fujino and Y. Gongyo (cf. \cite[4.13]{FG11}), implies \eqref{slc}  when the base $S$ is projective. Recent work of J. Koll\'ar on gluing lc centers provides us with a technique to prove the general case.

The absolute case of \eqref{lc} or the case when $S$ is projective is also known (cf. \cite{Gongyo10}).  However, the general relative case does not seem to be available anywhere in the literature (see the remark in \cite[4.16]{FG11}). 
\end{rmk}

\section{Hodge theoretic construction}

\begin{construction}[cf. {\cite[8.4.6]{Kollar07}}]\label{loc}
Let $(X,\Delta )$ be a log canonical pair, and $f:X\to Y$ a proper surjective morphism of normal varieties with connected fibers such that $n=\dim X-\dim Y$ and $K_X+\Delta \sim _{\Q ,Y}0$. 

 Let $p:W\to X$ be a log resolution of $(X,\Delta)$.
Write
$$p^*(K_{X}+\Delta)=K_W+E+F-G,$$ where $E$ and $G$ are integral effective divisors with no common components and $F=\{F\}$. Let $a$ be an integer such that $aF$ is an integral divisor. Denote by $\phi=f \circ p:W\to Y$.

Let  $Y^0\subset Y$ be a smooth open subset such that $\phi$ is smooth over $Y^0$. We denote by $\bullet^0$ the base change over $Y^0$.
Replacing  $Y^0$ by an open subset, we may assume that $(K_{X}+\Delta)|_{X^0}\sim_{\mathbb{Q}} 0$. We define a line bundle  $V^0=\omega^{-1}_{W^0}(G^0-E^0)$ so that ${V^0}^{\otimes a}\cong \mathcal{O}_{W^0}(aF^0)$.
 This data defines a local system $\mathbb{V}^0$ on $W^0\setminus {\rm Supp}(E^0\cup F^0)$ (cf. \cite[8.4.6]{Kollar07}, \cite[3.2]{EV92} and its proof). 
 
  Consider  the normalization of the corresponding  $\mu_a$-cover $\pi:W'\to W^0$, and denote by $E'$ the reduced divisor supported on $\pi^*E^0$. The push-forward $\pi_*(\C|_{W'\setminus E'})$ has a $\mu_a$-action. If we decompose 
  $$\pi_*(\C|_{W'\setminus E'})=\bigoplus_i \pi_*(\C|_{W'\setminus E'})^{(i)}$$ into the corresponding eigenspaces, then  $\V^0$ is isomorphic to the restriction of $\pi_*(\C|_{W'\setminus E'})^{(1)}$ to $W^0\setminus (E^0\cup F^0)$, we denote $\pi_*(\C|_{W'\setminus E'})^{(1)}$ by $\mathbb{V}$. The choice of $\mathbb{V}$ is determined up to the choice of a unit in $\mathcal O_{Y^0}$. However, $(R^{n}\phi_*\mathbb{V})^{\otimes a}$ (we will sloppily denote $\phi|_{W^0\setminus E^0} $ by $\phi$) is a well defined local system on $Y^0$ (cf. \cite[8.4.7]{Kollar07}). 

  Denote by $\phi':W'\to Y^0$ the composite morphism  $\phi \circ\pi$ (and its restriction to open subsets). Then 
  $R^{n}\phi'_*\mathbb{V}$  is a direct summand of $R^{n}{\phi'}_*\mathbb{C}|_{W'\setminus E'}$ which carries a variation of mixed Hodge structure. We remark that even though $W'$ has quotient singularities, locally $(W',E')$ is a finite quotient of a log smooth pair and hence Hodge theoretically it behaves  as a smooth variety with a simple normal crossing divisor, at least for $\mathbb{Q}$-coefficients (cf. \cite[Section 1]{Steenbrink77}).
  
 The local system  $R^{n}\phi _*\mathbb{V}$ gives a variation of mixed Hodge structure on $Y^0$.  It follows from \cite[1.18]{Steenbrink77} that the bottom piece of the Hodge filtration gives an isomorphism
 $$F^nR^n{\phi'}_*(\mathbb{C}|_{W'\setminus E'})\cong {\phi'}_*\omega_{W'/Y}(E'). $$ 
 Considering the eigenspace, we conclude  that $F^nR^n{\phi}_*(\mathbb{\V}|_{W^0\setminus E^0})\cong {\phi}_*\mathcal{O}_{W^0}(G^0)$ is a line bundle over $Y^0$ (cf. \cite[8.4.5(7')]{Kollar07}).

Denote by $L$ this line bundle. 
Since $L$ is 1 dimensional, there exists an integer $i$ such that, $L\subset W_{n+i} ( R^{n}\phi_*\mathbb{V})$ but $L\not\subset W_{n+i-1} ( R^{n}\phi_*\mathbb{V})$. Let $\mathbb{H}$ be the smallest pure sub-$\mathbb{Q}$-VHS of  $Gr^W_{n+i}  (R^{n}\phi_*\mathbb{V})$ which contains $L$. 

\end{construction}

\begin{lemma}\label{l-ind}
 $\mathbb{H}$ does not depend on the choice of the resolution of $W$.
\end{lemma}

\begin{rmk}\label{r-loc}
 Let $\mathbb{H}_1$ and $\mathbb{H}_2$ be two local system defined over nonempty open subsets $U_1,U_2\subset Y$ such that $\mathbb{H}_1|_{U_1\cap U_2}\cong \mathbb{H}_2|_{U_2\cap U_2} $, since $Y$ is normal, then there is a unique local system $\mathbb{H}$ (up to a unique isomorphism) defined over $U_1\cup U_2$, such that its restriction to $U_i$ is isomorphic to $\mathbb{H}_i$.
\end{rmk}

\begin{proof}[Proof of \eqref{l-ind}]From the above remark, it suffices to verify the lemma for any nonempty open set $Y^0\subset Y$.
 For two resolutions $W_1,W_2$, we can assume that there exists a morphism $\psi:W_2\to W_1$. By shrinking $Y^0$, we can also assume that $W_1$ and $W_2$ are both smooth over $Y^0$ (cf. \eqref{r-loc}). 	 If $p_1^*(K_{X}+\Delta)=K_{W_1}+E_1+F_1-G_1$, and $aF_1$ is integral, then 
 $$p_2^*(K_{X}+\Delta)=\psi^*(K_{W_1}+E_1+F_1-G_1)=K_{W_2}+E_2+F_2-G_2,$$
 and $aF_2$ is easily seen to be integral. Applying the construction \eqref{loc} for both $W_i$  (here we choose the same unit in $\mathcal{O}_{Y^0}$ for $W_i$), we have that $W'_2$ is the normalization of $W^0_2\times_{W^0_1} W_1'$
 and we denote by $\psi':W_2'\to W_1' $.  Let $E^*={\rm Supp}(\psi'^{-1}{E}_1')$, then we know that $E_2'\subset E^*$.  
 \begin{diagram}
 E_2'\subset E^*\subset W_2'& \rTo& E_2\subset W_2 \\
 \dTo^{\psi'} & &\dTo^{\psi}\\
 E_1'\subset W_1'& \rTo& E_1\subset W_1 
 \end{diagram}
 
 Therefore, there exist morphisms between mixed Hodge structures on $Y^0$
  $$i:R^n{\phi'_2}_*(\mathbb{C}|_{W'_2\setminus E_2'})\to R^n{\phi'_2}_*(\mathbb{C}|_{W'_2\setminus E^*})\qquad \mbox{and} \qquad$$
$$j:  R^n{\phi'_1}_*(\mathbb{C}|_{W'_1\setminus E_1'})\to R^n{\phi'_2}_*(\mathbb{C}|_{W'_2\setminus E^*}) .$$
 Applying the Hodge filtration $F^n(\cdot)$, it follows from \cite[1.18]{Steenbrink77} that 
 $$F^nR^n{\phi'_2}_*(\mathbb{C}|_{W'_2\setminus E_2'})\cong {\phi'_2}_*\omega_{W_2'/Y}(E_2'),  $$
 and similarly for other pairs.
Thus we have morphisms $$F^ni: {\phi'_2}_*\omega_{W_2'/Y}(E_2')\to{\phi'_2}_*\omega_{W_2'/Y}(E^*) \qquad \mbox{and} $$
$$\qquad \qquad \qquad F^nj: {\phi'_1}_*\omega_{W_1'/Y}(E_1') \to{\phi'_2}_*\omega_{W_2'/Y}(E^*).$$
 which are isomorphism by \eqref{l-ho}.

Replacing $Y^0$ by a smaller open set, by the discussion in \eqref{loc}, there are morphisms between mixed Hodge structures, $$R^n {\phi_1 }_* \mathbb{V}_1\rightarrow  R^n{\phi_2}_*(\mathbb{C}|_{W'_2\setminus E^*})\leftarrow R^n{\phi_2}_*\mathbb{V}_2.$$
We conclude that if we take $F^n(\cdot)$ of each term above, then we obtain an induced isomorphism $F^n(R^n {\phi_1 }_* \mathbb{V}_1)\to F^n(R^n {\phi_2 }_* \mathbb{V}_2)$.
Since polarized-VHSs form a semi-simple category, considering $Gr^W_{n+i}R^n{\phi _j}_*\mathbb V _j$, we see that the images of $\mathbb{H}_1$ and $\mathbb{H}_2$ are mapped to the same pure Hodge substructure of $Gr^W_{n+i} R^n{\phi_2}_*(\mathbb{C}|_{W'_2\setminus E^*})$. 

\end{proof}

\begin{lemma}\label{l-ho}
Let $f:Y_2\to Y_1$ be a birational morphism between normal projective varieties. Let $\Delta_1$ be reduced divisor on $Y_1$, such that $(Y_1,\Delta_1)$ is log canonical. Let $f^{-1}_*(\Delta_1)\leq \Delta_2\leq f^{-1}_*(\Delta_1)+{\rm Ex}(f)$ be an effective Weil divisor on $Y_2$ whose support contains all divisors of discrepancy $-1$ with respect to the pair $(Y_1,\Delta_1)$. Then there is a natural morphism $f_*\mathcal O_{Y_2}(K_{Y_2}+\Delta _2)\to \mathcal O_{Y_1}(K_{Y_1}+\Delta _1)$ inducing an isomorphism
$$H^0(Y_2,\mathcal O_{Y_2}(K_{Y_2}+\Delta_2))\cong H^0(Y_1,\mathcal O_{Y_1}(K_{Y_1}+\Delta_1)).$$
\end{lemma}
\begin{proof}
By assumption, we can write
$$F+f^*(K_{Y_1}+\Delta_1)=K_{Y_2}+\Delta_2+E,$$
where $F,E\ge 0$, $F$ is exceptional and $\lfloor E \rfloor=0$. Therefore, 
$$H^0(Y_1, \mathcal O_{Y_1}(K_{Y_1}+\Delta_1))\cong H^0(Y_2, \mathcal O_{Y_2}(\lfloor F+f^*(K_{Y_1}+\Delta_1)\rfloor ))\qquad $$ 
$$\qquad \cong H^0(Y_2, \mathcal O _{Y_2}(K_{Y_2}+\Delta_2+\lfloor E\rfloor))\cong H^0(Y_2, \mathcal O_{Y_2}(K_{Y_2}+\Delta_2)).$$
\end{proof}

\section{Finiteness of {\bf B}-representations}\label{s-2}
In this section we prove \eqref{t-fb}.
By assumption, there exists a positive integer $m$, such that $|m(K_{X}+\Delta)|$ is base point free and
it  induces an algebraic fibration structure $f:X\to Y$, so that 
$$Y={\rm Proj}\bigoplus_{d\geq 0} H^0(X,\mathcal O_{X}(dm(K_{X}+\Delta))).$$
It follows from Kawamata's canonical bundle formula (cf. \cite{Kawamata98}) that we can write
$$K_{X}+\Delta\sim_{\mathbb{Q}} f^*(K_Y+B+J),$$
where $B$ is the boundary part and $J$ is the moduli part. Let $P\in Y$ be a codimension 1 point, then the coefficient of $P$ in $B$ is defined by $$1-{\rm lct}(X,\Delta;f^{-1}(P))$$ where the log canonical threshold is computed over a  neighborhood of the generic point of $P$. In particular, $B$ is an effective $\mathbb{Q}$-divisor. The moduli part $J$ is defined as an equivalence class of $\mathbb{Q}$-divisors, coming from Hodge theory. 

For any $g\in {\rm Bir}(X, \Delta)$ and any $d\geq 0$ we have homomorphisms $\rho _{dm}: {\rm Bir}(X,\Delta )\to {\rm Aut }_\C H^0(X,\mathcal O_{X}(dm(K_{X}+\Delta)))$ where
$$\rho _{dm}(g)=g^*: H^0(X,\mathcal O_{X}(dm(K_{X}+\Delta)))\to H^0(X,\mathcal O_{X}(dm(K_{X}+\Delta))).$$
Consider the induced homomorphism 
$$\chi:{\rm Bir}(X,\Delta)\to {\rm Aut}(Y).$$ 
For any element $g\in {\rm Bir}(X, \Delta)$, there is a commutative diagram
\begin{diagram}
 X&\rDashto^g &X\\
\dTo^f& &\dTo^f\\
Y&\rTo^{\chi(g)}&Y
\end{diagram}
Let $F$ be the generic geometric fiber of $f$ and $n={\dim}F={\dim} X-{\dim} Y$ be the relative dimension.
Then for any $d$, we have a short exact sequence
$$1\to G\to \rho_{dm}({\rm Bir}(X,\Delta))\to \chi({\rm Bir}(X,\Delta))\to1,$$
where $G\subset H^0(Y,\mathcal{O}^*_Y)=\C^*$.  We know  that $G$ is finite (cf. \cite[4.9]{Gongyo10}), so it suffices to show the following result.

\begin{theorem}\label{t-f-image}
The image $\chi({\rm Bir}(X,\Delta))\subset {\rm Aut}(Y) $
is finite.
\end{theorem}
First it is easy to see the following result.
\begin{lemma}\label{l-des}
The image of $\chi( {\rm Bir}(X,\Delta))$ is contained in ${\rm Aut}(Y,B)$.
\end{lemma}
\begin{proof}
Let $P\in Y$ be a codimension 1 point, as we noted, the coefficient of $P$ in $B$ is defined by $1-{\rm lct}(X,\Delta;f^{-1}(P))$. This number is unchanged if we replace $(X,\Delta )$ by any log resolution. Thus the lemma easily follows from the above commutative diagram, and the fact that $g$ is in ${\rm Bir}(X,\Delta)$.
\end{proof}

By the following result, we see that the birational maps also preserve $\mathbb{H}$, where  $\mathbb{H}$ is the local system defined in \eqref{loc},  
\begin{prop}\label{l-lsyt}
(1) Let $g\in {\rm Bir}(X,\Delta)$. If we assume $\mathbb{H}$ is defined over an open set $Y^0\subset Y$, then over $Y^0\cap\chi (g)^{-1} (Y^0) $, there exists an isomorphism $i_g: \chi (g)^*\mathbb{H}\cong \mathbb{H}$. 

\noindent (2) Let $g_1,g_2\in {\rm Bir}(X,\Delta)$, then over  $Y^0\cap \chi (g_1)^{-1} (Y^0) \cap \chi (g_2)^{-1} (Y^0)$ we have $i_{g_1}\circ i_{g_2}=i_{g_1\circ g_2}$.
\end{prop}
\begin{proof}
(1) Let $W$ be a common log resolution,
\begin{diagram}
& &W & &\\
&\ldTo^{p_1}& &\rdTo^{p_2}&\\
 X& &\rDashto^g& &X\\
\dTo^f& & & &\dTo^f\\
Y& &\rTo^{\chi (g)}& &Y
\end{diagram}

Since $p_1^*(K_X+\Delta)=p_2^*(K_{X}+\Delta)$,  we obtain the same local system on some open subset $ W^0\subset W$, whose image is an open subset $Y^0\subset Y$.
We can shrink $Y^0$, so that $R^n {\phi_1}_*(\mathbb{V})$ (resp. $R^n {\phi_2}_*(\mathbb{V})$) is defined over $\chi(g)^{-1} (Y^0) $ (resp. $Y^0$), where we denote $f\circ  p_i$ by $\phi_i$.

Since $p_1^*(K_{X}+\Delta)=p_2^*(K_{X}+\Delta)$, we conclude $\chi(g)^*(L)\cong L$.
Therefore, because of the semi-simplicity of VHS, the isomorphism $$\chi(g)^{*} :R^n{\phi_1}_*(\mathbb{V})\to R^n{\phi_2}_*(\mathbb{V})$$ which sends $L$ to $L$, will send $\mathbb{H}$ to $\mathbb{H}$.

(2) This also follows by a similar argument since $(\chi (g_1\circ g_2))^{*}L=L$.
\end{proof}

 Since $\chi({\rm Bir}(X,\Delta))$ preserves the polarization on $Y$, it is 
a subgroup of ${\rm PGL}(N)$ for some $N$. Let $G$ be the closure of  $\chi({\rm Bir}(X,\Delta))$ in ${\rm PGL}(N)$, in particular $G$ is an algebraic group. Thus, to show that $\chi({\rm Bir}(X,\Delta)$ is finite, 
we only need to verify that $G$ does not contain  $\G_a$ or $\G_m$.
We will show that $G$ does not contain $\G_m$. (The argument for $\mathbb{G}_a$ is similar and we leave it to the reader.)

Choose $Y_0\subset Y\setminus B$ to be an open set which the Hodge structure $\mathbb{H}$ does not degenerate, and define an open set
$$\tilde{Y}=\bigcup_{g\in \chi({\rm Bir}(X,\Delta))}  g(Y_0).$$ Therefore, $\tilde{Y}$ is invariant under the action of $G$ as it is invariant under the Zariski dense subgroup $\chi({\rm Bir}(X,\Delta))$ of $G$. By \eqref{l-lsyt} and \eqref{r-loc}, $\mathbb{H}$ is non-degenerate over $\tilde{Y}$.

Then for a general point $z\in \tilde{Y}$, we consider the closure of the orbit $o:\P^1\to Y$ of $\G_m\cdot z$. Since $\tilde{Y}$ is $G$-invariant, we know that the pull back $o^*\mathbb{H}$ is well defined over $\G_m\subset \P^1$.
Let $\pi:Y'\to Y$ be a $G$-equivariant resolution of $(Y,Y\setminus \tilde{Y})$, thus $\mathbb{H}'=\pi^*\mathbb{H}$ is defined outside a simple normal crossing divisor (see \cite{Kollar07b}). We can write $\pi^*(K_Y+B+J)=K_{Y'}+B'+J'$ where $(Y',B'+J')$ is the decomposition which the Kawamata subadjunction formula yields on $Y'$. In particular, $(Y',B')$ is sub log canonical. 

Let $\phi:Z\to Y'$ be  a branched covering such that $\phi^*\mathbb{H}'$ has unipotent monodromies and hence admits a canonical extension over $Z$ (see \cite[8.10.10]{Kollar07}). Let $\phi_C:C\to \P^1$ be the normalization of $\mathbb{P}^1\times_{Y} Z$. As $o(\P^1)$ is the compactification of a general orbit, we know that $\phi_C$ is of degree $d$. Then we know that $\phi_C^*(o^*\mathbb{H})$ has a canonical extension from the preimage of $\phi_C^{-1}(\mathbb{G}_m)$ to $C$ which coincides with the restriction of the canonical extension of $\phi^*\mathbb{H}'$ over $C$ (see \cite{Deligne70}, \cite[8.10.8]{Kollar07}). Therefore, if $J_Z=\bar{E}^{n+i,0}(\phi^*\mathbb{H}')$, then $$J_C=\bar{E}^{n+i,0}(\phi^*\mathbb{H}'|_{C})=\bar{E}^{n+i,0}(\phi^*\mathbb{H}') |_C=J_Z|_C.$$


 As $\mathbb{G}_m \cong \mathbb{A}^1\setminus\{0 \}\subset \mathbb{P}^1$. There exists a ramified cover $d:\P^1\to \P^1$, such that if we denote by $o_d=o \circ d$, then $o_d^*(\mathbb{H})$ is defined over $\A^1\setminus\{0\}$, and it has unipotent monodromies near $0$ and $\infty$. As $o^*_d(\mathbb{H})$ is non-degenerate on $\mathbb{P}^1\setminus\{0,\infty \}$,  $o_d^*\mathbb{H}$ is trivial by Deligne's  semi-simple theorem (see \cite[4.2.6]{Deligne71}). Since $\bar{E}^{n+i,0}(o_d^*(\mathbb{H}))=\mathcal O _{\mathbb P ^1}(J_d)$, we have
 $$\frac{1}{d}d_*J_d\sim_{\mathbb{Q}}\frac{1}{\deg\phi_C}\phi_{C*}J_C\sim _\Q J'|_{\mathbb{P}^1}=0,$$
 where the first equality follows from the fact that the moduli part is well defined for any choice of  base change such that the monodromy is unipotent, and since $\phi$ is finite the second equality follows from restricting both sides of $\frac{1}{{\rm deg}\phi}(\phi_*J_Z)\sim_{\mathbb{Q}} J'$ under the lifting $\mathbb{P}^1\to Y'$ of $o$.

  On the other hand, by \eqref{lm-logbig}, we have 
$$\deg(o^*(K_Y+B+J))=\deg(o^*(K_Y+B))\leq 0.$$
 Since $K_Y+B+J$ is ample, this is a contradiction and so this completes the proof of \eqref{t-f-image}. 


\begin{lemma}\label{lm-logbig}
Let $(Y,B)$ be a sub log canonical pair and  $\psi: \G_m\times (Y,B) \to (Y,B)$  a faithful action.  For general $t\in Y$, if we denote by $\psi_t:\P^1\times\{t\}\to Y$ the closure of the orbit, then  $\deg \psi_t^*(K_Y+B)\le0$.
\end{lemma}
\begin{proof}
Let $\pi:Y'\to Y$ be a log resolution of $(Y,B)$ and $B'={\rm Supp}( \pi^{-1}_*(B)+{\rm Ex}(\pi))$. 
We may assume that $\pi ({\rm Ex}(\pi))$ is contained in the union of
${\rm Supp}(B)$ and the singular locus of $Y$. 
Thus, for a general point $t\in Y$, we have that 
$$\mathbb G _m\cdot t \cap ({\rm Supp}(B)\cup {\rm \pi(Ex}(\pi)))=\emptyset.$$ We then have the induced morphisms $$(\mathbb A ^1-\{ 0 \})\to Y\setminus ({\rm Supp}(B)\cup {\rm Ex}(\pi)))\qquad {\rm and}\qquad \phi _t:\mathbb P ^1 \to Y'.$$
 Since $(Y,B)$ is log canonical, $\pi^*(K_{Y}+B)\le K_{Y'}+B'$.  We may assume that there is a smooth open set $T$ and generically finite morphism, $\phi_T: \P^1\times T\to Y'$ such that $\phi_T^{-1}(B')\subset \{0,\infty\}\times T$, and  $\phi_T|_{t\times \P ^1}=\phi_t$. Then it is easy to see by a local computation that
 $$\phi_T^*(K_{Y'}+B')\le K_{T\times \P^1} +T\times\{0\}+T\times \{\infty\}.$$
Therefore, we conclude that $$\deg \psi_t^*(K_Y+B)\le\deg  \phi _t^*(K_{Y'}+B')\leq \P^1\cdot (K_{T\times \P^1} +T\times\{0\}+T\times \{\infty\} ),$$ which is computed by $\deg (K_{\mathbb P^1}+\{0\}+\{\infty\})=0.$
\end{proof}

\section{Abundance for slc pairs}
In this section, we will study the semi-log canonical abundance. By \cite{Fujino00}, it is known that if $S$ is a point, then \eqref{t-fb} implies \eqref{slc}. By \cite{FG11},  \eqref{slc} is also known when $S$ is projective.
The gluing theory for log canonical centers, developed by J. Koll\'ar (cf. \cite[Chapter 3]{Kollar11}), provides a very powerful tool to study semi-log canonical varieties and allows us to prove \eqref{slc} in full generality.

More specifically, Koll\'ar's recent theory of giving any lc center  a source and a spring provides a useful tool for checking that the pro-finite equivalence relation is finite (cf. \cite{Kollar11b}).
We note that, this is not true for general pro-finite equivalence relations. The fact that we are gluing lc centers for relatively projective pairs plays an essential role here. 

First, the main theorem for abstract gluing theory is the following.

\begin{theorem}[{\cite[3.40]{Kollar11}}]\label{t-quotient}
 Let $(X, S_*)$ be an excellent scheme or algebraic space over a
field of characteristic 0 with a stratification. Assume that
$(X, S_{*})$ satisfies the conditions (HN) and (HSN). Let $R
\rightrightarrows X$ be a finite, set theoretic, stratified
equivalence relation. Then
\begin{enumerate}
\item the geometric quotient $X/R$ exists,
\item $\pi : X \to X/R$ is stratifiable and
\item $(X/R, \pi_*S_*)$ also satisfies the conditions (HN) and (HSN).
\end{enumerate}
\end{theorem}
We now recall the following notation, which was first introduced by F. Ambro (cf. \cite{Ambro03}).  

\begin{defn}\label{d-cls} We call $f:(X,\Delta)\to Y$ {\it a crepant log structure} if
\begin{enumerate}
\item $(X,\Delta)$ is log canonical,
\item $f$ is projective, surjective, with connected fibers,
\item $K_X+\Delta\sim_{\mathbb{Q},f}0 $.
\end{enumerate}
We note that, by taking a dlt modification of $(X,\Delta)$ (cf. \cite[3.1]{KK10}), we can assume that $(X,\Delta)$ is dlt.
\end{defn}

As in \cite[Section 3]{HX11}, we let $\bar g : \bar X\to \bar Y$ be the morphism over $S$ induced by the relatively semiample $\Q$-divisor $K_{\bar X}+\bar \Delta +\bar D$ and we consider the pro-finite equivalence relation 
$$(\sigma_1,\sigma_2): \bar{T}\rightrightarrows \bar{Y}$$
given by the image of the pro-finite relation $D^n\rightrightarrows \bar{X}$.  Recall that $D^n$ denotes the normalization of $\bar{D}$ on $\bar{X}$. Note that 
 $\bar{T}$ and $\bar{Y}$ are normal.   We may assume that $\bar g$ is induced by $|m(K_{\bar X}+\bar \Delta +\bar D)|$ for some sufficiently divisible positive integer $m$. 
If we write $K_{D^n}+\Theta=(K_{\bar{X}}+\bar{\Delta}+\bar{D})|_{D^n} $, then $(D^n,\Theta)$ is also log canonical.

As in \cite[Section 3]{HX11}, we consider the minimal qlc stratifications $S_*\bar{T}$ and $S_*\bar{Y}$ given by the crepant log structures $(D^n,\Theta)\to \bar{T}$ and $(\bar{X},\bar{\Delta}+\bar{D})\to \bar{Y}$. We note that if we let $(X^d,\Delta^d)$ be a dlt modification of 
$(\bar{X},\bar{\Delta}+\bar{D})$ (cf. \cite[3.1]{KK10}), then $$(X^d,\Delta^d)\to (\bar{X},\bar{\Delta}+\bar{D})\to \bar{Y}$$ gives the same minimal qlc stratification on $\bar Y$.

\begin{prop}We have the following facts.
\begin{enumerate} 
\item The stratifications $S_*\bar{T}$ and $S_*\bar{Y}$ satisfy conditions (HN) and (HSN).
\item The pro-finite relation  $(\sigma_1,\sigma_2):\bar{T}\rightrightarrows \bar{Y}$ is stratified.
\end{enumerate}
\end{prop}
\begin{proof}The first statement follows from \cite[5.7]{KK10} and the second one follows from \cite[3.11]{HX11} which is a consequence of \cite[1.7]{KK10}. 
\end{proof}

By \eqref{t-quotient}, to prove that the quotient $\bar Y/\bar T$ exists, we only need to verify that $\bar{T}\rightrightarrows\bar{Y}$ generates a finite relation $R\to \bar{Y}$. By \cite[3.61]{Kollar11}, it suffices to check this over the generic point of each stratum $V^0$ of $S_i\bar{Y}$.
 We work over the generic point of $\bar f (V^0)$, where $\bar f:\bar Y\to S$ is the induced morphism. 
Let $\eta\in S$ be one such point and $\bar{\eta}$ its algebraic closure. We must verify that the equivalence relation $R\times_S\bar\eta\subset\bar{Y}\times \bar{Y}\times_S \bar\eta$ is finite over $V^0\times\bar{\eta}$. 

Let $Z$ be one of the minimal lc centers of $(X^d,\Delta^d)$ which dominates the closure $V$ of $V^0$, then $(Z,{\rm Diff}^*_{Z}\Delta^d)$ is dlt and its restriction over $V^0$ is klt. We take the Stein factorization $Z \to \tilde{V}\to V $. We will need to following results from \cite{Kollar11b}.
\begin{prop}[{\cite[1]{Kollar11b}}]\label{p-equ}
For different choices of the minimal non-klt centers $Z$ and $Z'$, we have
\begin{enumerate}
\item $(Z,{\rm Diff}^*_{Z}\Delta^d)$ is {\bf B}-birational to $(Z',{\rm Diff}^*_{Z'}\Delta^d)$
over $V$, and
\item $\tilde{V}$ is isomorphic to $\tilde{V}'$ over $V$.
\end{enumerate}
\end{prop}

\begin{defn}[{\cite[19]{Kollar11b}}]
We denote by ${\rm Src}(V,X^d,\Delta^d):=(Z,{\rm Diff}^*_{Z}\Delta^d)$ a {\it source} of the lc center $V $ and ${\rm Spr}(V,X^d,\Delta^d):=\tilde{V}$ its {\it spring}. As in \eqref{p-equ}, ${\rm Src}(V,X^d,\Delta^d)$ is determined up to {\bf B}-equivalence.
\end{defn}

Let $V^0_{ij}$ be an $i$-dimensional stratum, and $\tilde{V}^0_{ij}$ the pre-image of $V^0_{ij}$ under the morphism $\tilde{V}_{ij}\to V_{ij}$.  
We define ${\rm Spr}_i(\bar Y, X^d,\Delta ^d)=\amalg_{j} \tilde{V}^0_{ij} $, where the disjoint union runs over all $i$-dimensional strata $V^0_{ij}$.
Then the main structural result is the following.
\begin{thm}[{\cite[33]{Kollar11b}}]
Let $R\subset \bar{Y}\times \bar{Y}$ be the relation generated by $\bar{T}\rightrightarrows \bar{Y}$ as above. Let $p_i: {\rm Spr}_i(\bar{Y}, X^d,\Delta^d)\to S_i\bar{Y}$ be the induced finite morphisms. Let $\bar{\eta}_{ij}$ be the algebraic closure of the generic point of $\bar f (V_{ij})$. 
Then 
 $$((p_i\times p_i)^{-1}(R\cap (S_i\bar{Y}\times S_i\bar{Y})) \times_S \bar{\eta}_{ij}) \cap( \tilde{V}^0_{ij}\times \tilde{V}^0_{ij} \times_S \bar{\eta}_{ij})$$ 
is a subset of the graph $\cup_g \Gamma(\chi(g))$ for all $g\in {\rm Bir} (Z_{\bar{\eta}_{ij}} ,{\rm Diff}^*_{Z_{\bar{\eta}_{ij} }}\Delta^d)$.
\end{thm}

\begin{proof}[Proof of \eqref{slc}]By \eqref{t-f-image}, we have that $\cup_g \Gamma(\chi(g))$ is finite over $\tilde{V}^0_{ij}\times \bar{\eta}_{ij}$, and so the hypotheses of \eqref{t-quotient} are satisfied. Thus there exists a quotient $Y$ 
of $\bar{T}\rightrightarrows \bar{Y}$. Following the proof of \cite[3.1]{HX11},  we see that there exists a line bundle $L$ on $Y$ whose pull back to $X$ is isomorphic to $\mathcal O _X(m(K_X+\Delta))$ for some integer $m>0$. Then it is easy to see that $L$ is relatively ample over $S$, which means that $K_X+\Delta$ is relatively semi-ample over $S$.\end{proof}

\begin{proof}[Proof of \eqref{c-lc}]The argument is the same as in \cite[7.4]{KMM94}. We include it here for reader's convenience. It follows from  \eqref{slc} that $m(K_T+\Delta _T):=m(K_X+\Delta)|_T$ is base point free over $S$ for any integer $m>0$ sufficiently divisible. By assumption (3), we may assume that 
the relative base locus of $m(K_X+\Delta )$ is contained in the support of $T$. We write
$$m(K_X+\Delta)-T=(m-1)(K_X+\Delta-\epsilon P)+K_X+\Delta-(T-(m-1)\epsilon P).$$ 
 Since $K_X+\Delta-\epsilon P$ is semi-ample over $S$, and $(X,\Delta-(T-(m-1)\epsilon P))$ is klt for  $0< \epsilon\ll 1$, 
by Koll\'ar's injectivity theorem (cf. \cite{Kollar86}),  we have that 
\begin{diagram}
R^1f_*\mathcal O_X(m(K_X+\Delta)-T)&\rTo^{T}& R^1f_*\mathcal O_X(m(K_X+\Delta)) 
 \end{diagram}
is an injection and hence $f_*\mathcal O_X(m(K_X+\Delta))\to f_*\mathcal O_T(m(K_T+\Delta_T))$ is surjective. 
We have a commutative diagram.
\begin{diagram}
f^*f_*\mathcal O_X(m(K_X+\Delta))&\rTo & f^*f_*\mathcal O_T(m(K_T+\Delta_T))\\
\dTo & & \dTo\\
\mathcal{O}_X(m(K_X+\Delta))&\rTo&\mathcal{O}_T(m(K_T+\Delta _T))
\end{diagram} 
Since the upper arrow and the right arrow are surjective,  $m(K_X+\Delta)$ is relatively globally generated along $T$ over $S$ and so  $m(K_X+\Delta)$ is relatively globally generated over $S$. 
\end{proof}

\begin{proof}[Proof of \eqref{lc}]By \eqref{slc}, we may assume that $X$ is irreducible, i.e., we only need to treat the case that $(X,\Delta)$ is log canonical.
 We may assume that $S$ is normal and dominated by $X$. By induction, we may assume that $\eqref{lc}$ is known for lower dimensions. Replacing $(X,\Delta)$ by a dlt modification, we may assume that $(X,\Delta)$ is $\Q$-factorial and dlt. Let $T=\lfloor \Delta  \rfloor$. Write $(K_X+\Delta)|_{T}=K_{T}+\Delta_T$, then $(T,\Delta _T)$ is a dslt pair. 

We run a $(K_{X}+\Delta-\epsilon T)$-MMP with scaling by an ample divisor $H$ for a small positive integer $\epsilon$.
If $T$ has a component which dominates $S$, then the MMP ends with a Fano contraction, and we can apply the same argument as in \cite{Gongyo10}. 
Otherwise, all components of $T$ are vertical over $S$.   
Since over the generic point of $S$, we have that $K_{X}+\Delta$ is klt, then we know that it is $\Q$-linearly equivalent to $0$ (cf. \cite[V, 4.9]{Nakayama04}). 
 
Therefore, $(X,\Delta-\epsilon T)$ is a klt pair, whose generic fiber has a good model.
By, \cite[1.1 and 2.9]{HX11}, we conclude that the MMP with scaling terminates with a good model $X'$ of $(X,\Delta-\epsilon T)$ over $S$. Since $X'$ is indeed a minimal model for all $(X,\Delta-\epsilon' T)$ with $0<\epsilon'<\epsilon$, which is  good again by \cite[1.1]{HX11}. 

The MMP sequence is $(K_X+\Delta)$-trivial, so $(X',\Delta')$ is a lc pair. We let $\mu:(X^d,\Delta ^d)\to (X',\Delta ')$ be a dlt model. Let $T'$ be the strict transform of $T$ on $X'$,
then $(X', \Delta ' -\epsilon T')$ is klt and the non-klt locus of $(X',\Delta ')$ is contained in ${\rm Supp }(T')$. Since each $\mu$-exceptional divisor has coefficient $1$ in $\Delta ^d$, it follows that $\lfloor \Delta ^d \rfloor ={\rm Supp}(\mu ^* T')$. Thus if $T^d=\mu ^* T'$, then $K_{X^d}+\Delta ^d-\epsilon T^d=\mu ^*(K_{X'}+\Delta '-\epsilon T')$ is semiample over $S$. By induction on the dimension, $K_{\Sigma} +\Delta_{\Sigma} $ is semiample over $S$ for all components $\Sigma $ of $ \lfloor \Delta^d \rfloor={\rm Supp}(T^d)$.
By \eqref{c-lc}, $K_{X^d}+\Delta^d$ is semi-ample over $S$ and hence so is $K_{X'}+\Delta'$. Since each step of the MMP is $(K_{X}+\Delta)$-trivial, this implies that $K_X+\Delta$ is semi-ample over $S$, i.e., $K_X+\Delta\sim_{\mathbb{Q},S}0$.  
\end{proof}

 \end{document}